\documentclass[12pt,leqno]{article}

\usepackage[T1]{fontenc}
\usepackage{amsmath,amsthm,mathtools}

\IfFileExists{newtxtext.sty}{%
  \usepackage{newtxtext}%
  \usepackage{newtxmath}%
}{%
  \usepackage{mathptmx}\usepackage{amssymb}%
}
\usepackage[dvipsnames]{xcolor}

\usepackage{iftex}
\ifPDFTeX\usepackage[utf8]{inputenc}\fi
\usepackage[a4paper,margin=1in]{geometry}
\usepackage{aliascnt}
\usepackage{enumitem}
\usepackage{needspace}
\usepackage[protrusion=true,expansion=false]{microtype}
\usepackage[hyperfootnotes=false,colorlinks=true,linkcolor=blue,citecolor=blue,
  urlcolor=blue]{hyperref}
\usepackage{comment}

\makeatletter
\newcommand\blfootnote[1]{%
  \begingroup
  \renewcommand\@makefntext[1]{\noindent\strut ##1}%
  \renewcommand\thefootnote{}\footnote{#1}%
  \addtocounter{footnote}{-1}%
  \endgroup
}
\makeatother

\setlist{itemsep=0.25em,topsep=0.4em}
\newtheorem{theorem}{Theorem}[section]
\newaliascnt{proposition}{theorem}
\newtheorem{proposition}[proposition]{Proposition}
\aliascntresetthe{proposition}
\newaliascnt{lemma}{theorem}
\newtheorem{lemma}[lemma]{Lemma}
\aliascntresetthe{lemma}

\newaliascnt{corollary}{theorem}

\aliascntresetthe{corollary}
\newaliascnt{conj}{theorem}

\aliascntresetthe{conj}
\theoremstyle{definition}
\newtheorem*{problem}{Problem}
\newaliascnt{definition}{theorem}

\aliascntresetthe{definition}

\newaliascnt{remark}{theorem}
\newtheorem{remark}[remark]{Remark}
\aliascntresetthe{remark}

\newaliascnt{example}{theorem}

\aliascntresetthe{example}

\usepackage[nameinlink,capitalise,noabbrev]{cleveref}

\newcommand{\R}{\mathbb R}
\newcommand{\eps}{\varepsilon}
\newcommand{\ip}[2]{\langle #1,#2\rangle}
\newcommand{\Sc}{\mathcal R}

\newcommand{\Z}{\mathbb Z}

\numberwithin{equation}{section}
\hypersetup{pdftitle={Positive biorthogonal curvature on $S^2\times T^2$},pdfauthor={Zhiqi Chen and Hui Zhang}}
\title{\textbf{Positive biorthogonal curvature on $S^2$-bundles over  closed  surfaces}}
\author{Zhiqi Chen and Hui Zhang}
\date{}
\begin{document}
\raggedbottom
\clubpenalty=10000
\widowpenalty=10000
\displaywidowpenalty=10000
\maketitle
\begin{abstract}
In this paper, we  construct an explicitly  one-parameter family of Riemannian metrics with positive biorthogonal
curvature on $S^2\times T^2$,  answering a question raised by Bettiol.    The   method also yields  metrics of positive biorthogonal curvature
on every smooth oriented $S^2$-bundle over a closed oriented surface. Since
$\pi_1(S^2\times T^2)=\mathbb Z^2$ is not virtually free, these examples
separate positive biorthogonal curvature from positive isotropic curvature
at the level of fundamental groups.
\end{abstract}
\blfootnote{%
}

\section{Introduction}\label{sec:introduction}
 Let $(M,g)$ be a Riemannian four-manifold. If
$\sigma\subset T_pM$ is a two-plane, its biorthogonal curvature is the
average
\[
 K_g^\perp(\sigma)=\frac12\bigl(K_g(\sigma)+K_g(\sigma^{\perp_g})\bigr),
 \qquad k_g(p)=\min_{\sigma\subset T_pM}K_g^\perp(\sigma),
\]
where $K_g$ denotes sectional curvature and $\sigma^{\perp_g}$ is the
orthogonal complement of $\sigma$ with respect to $g$.  In dimension four
this complement is again a two-plane, which is what makes the average
special to that dimension. We say that $g$ has \emph{positive biorthogonal
curvature} when $k_g>0$ everywhere. Since the scalar curvature is
four times the sum of the biorthogonal curvatures of the three
complementary pairs of coordinate planes of an orthonormal frame,  one has
\[
 K_g>0\quad\Longrightarrow\quad k_g>0\quad\Longrightarrow\quad s_g>0,
\]
where $s_g$ denotes scalar curvature. Positive biorthogonal curvature
nevertheless allows some negative sectional curvature, and the standard
Riemannian product on $S^3\times S^1$ shows that it does not imply
positive Ricci curvature. The condition therefore sits between two
rather different existence problems.

The systematic study of such curvature averages goes back to Seaman's
orthogonal pinching theorems~\cite{Seaman91}. Costa~\cite{Costa12} and
Costa--Ribeiro~\cite{CR14} related the function $k_g$ to a modified Yamabe
invariant and proved pinching and rigidity theorems. On the existence side,
Bettiol~\cite{Bettiol14} produced metrics of positive biorthogonal
curvature on $S^2\times S^2$ by following a Cheeger deformation with a
conformal change. Later, he proved that the condition is preserved under
connected sums and determined the closed simply connected four-manifolds
that admit it up to homeomorphism~\cite{Bettiol17}, using Hoelzel's
surgery stability criterion~\cite{Hoelzel}. In that classification, the
answer agrees with the positive scalar curvature problem once a change of
smooth structure is allowed, and the qualification about the smooth
structure is essential. We note that more recently, Brendle and Hung~\cite{BH} have
announced a metric of positive sectional curvature on $S^2\times S^2$,
obtained from a higher-order perturbation of a Cheeger--M\"uter metric.

Beyond the simply connected case, the first manifold on which the
existence problem becomes critical is $S^2\times T^2$. Bettiol singled it
out as a test case in \cite[Section~3]{Bettiol17}, and the same question is recorded by
Morgan and Pansu in~\cite[Question~4]{MP}.

\begin{problem}
Does $S^2\times T^2$ admit a Riemannian metric with positive biorthogonal
curvature?
\end{problem}
Two observations explain why this example is delicate. First, the natural
candidate fails at every point. The round--flat product metric $g_S+g_T$
has $k\equiv0$, because complementary mixed planes are both flat, and the
failure persists under conformal change --- no metric conformal to
$g_S+g_T$ has positive biorthogonal curvature
(see Remark~\ref{rem:product-conformal}). Any construction must therefore
leave the product conformal class before a conformal correction can help.
Second, the nearby curvature conditions do exclude this manifold. For
positive isotropic curvature, a condition introduced by Micallef and
Moore~\cite{MM88} which also implies positive scalar curvature,
Fraser~\cite{Fraser03} proved that in dimensions at least five the
fundamental group of a closed manifold contains no subgroup isomorphic to
$\Z\oplus\Z$, and in dimension four the classification of Chen, Tang and
Zhu~\cite[Corollary~1]{CTZ} forces the fundamental group to be virtually
free, that is, to contain a free subgroup of finite index. Since
$\pi_1(S^2\times T^2)=\Z^2$ is not virtually free, a metric of positive
biorthogonal curvature on $S^2\times T^2$ separates the biorthogonal
condition from the isotropic one at the level of fundamental groups. The
purpose of this paper is to construct such metrics explicitly.

\subsection{Explicit metrics on $S^2\times T^2$}

Write $S^2\subset\mathbb R^3$ for the unit sphere with round metric
$g_S$, and let $E_1,E_2,E_3$ be the standard basis of $\mathbb R^3$.
The vector fields and height function
\[
 X_i(p)=E_i\times p\quad(i=1,2,3),\qquad z(p)=\langle p,E_3\rangle
\]
will describe the connection and conformal factor. Fix a rank-two
lattice $\Lambda\subset\mathbb R^2$ and give
$T^2=\mathbb R^2/\Lambda$ its induced flat metric $g_T$.
The Euclidean coordinate fields $\partial_{t_1},\partial_{t_2}$
and their dual forms $dt_1,dt_2$ descend to this quotient, although
the coordinates $t_i$ themselves need not define functions on it.
For $\eps>0$, define a metric $g_\eps$ by
\begin{equation}\label{eq:metric}
 g_\eps\bigl((V,a,b),(V,a,b)\bigr)
 =|V-\eps aX_1-\eps bX_2|_{g_S}^2+a^2+b^2,
\end{equation}
where $(V,a,b)$ denotes the tangent vector
$(V,a\partial_{t_1}+b\partial_{t_2})$ at $(p,t)\in S^2\times T^2$.
This quadratic form determines a smooth symmetric two-tensor by
polarization, and the invertibility of $(V,a,b)\mapsto
(V-\eps aX_1-\eps bX_2,a,b)$ makes positive definiteness immediate.

\begin{theorem}\label{thm:main}
For every rank-two lattice $\Lambda\subset\mathbb R^2$ and every
$0<\eps\le1/4$, set
\begin{equation}\label{eq:conformal-factor}
 u_\eps=1+\frac{\eps^4z^2}{24(1+\eps^2)},
 \qquad \widehat g_\eps=u_\eps^2g_\eps.
\end{equation}
Then, at every point of $S^2\times(\mathbb R^2/\Lambda)$,
\begin{equation}\label{eq:final-bound}
 k_{\widehat g_\eps}
 \ge\frac{\eps^4}{u_\eps^2}
       \left(\frac{1}{12u_\eps}-\frac{\eps^2}{2}\right)
 \ge\frac{\eps^4}{384}>0.
\end{equation}
Moreover, $\widehat g_\eps$ is invariant under translations of the
torus and converges to $g_S+g_T$ in the $C^\infty$ topology as
$\eps\to0$.
\end{theorem}

For example, if $dp$ is the differential of the inclusion
$S^2\subset\mathbb R^3$, the choice $\eps=1/4$ gives the metric
\begin{equation}\label{eq:explicit}
 \widehat g=
 \left(1+\frac{z^2}{6528}\right)^2
 \left(\left|dp-\tfrac14X_1\,dt_1-\tfrac14X_2\,dt_2\right|^2
              +dt_1^2+dt_2^2\right),
\end{equation}
whose minimum biorthogonal curvature is at least
$1/98304=(1/4)^4/384$, the case $\eps=1/4$ of~\eqref{eq:final-bound}.
The norm in this formula is the Euclidean norm on
$\mathbb R^3$-valued one-forms.

The connection deformation in~\eqref{eq:metric} is needed before any
conformal correction can succeed, because by
Remark~\ref{rem:product-conformal} the conformal class of the round--flat
product contains no metric with $k_g>0$ at all. The noncommuting rotations
$X_1,X_2$ introduce a term of order $\eps^4$ whose leading part is
positive away from the poles and has positive average on the
sphere. The factor $u_\eps$ redistributes this term so that it is
positive at the poles as well. This use of a conformal change after
a geometric deformation is related to Bettiol's construction
\cite{Bettiol14}. Here the formulas give a quantitative estimate
on a prescribed manifold with infinite fundamental group.

\subsection{Sphere bundles over surfaces}

The same conformal method applies to a prescribed sphere bundle over a
surface, provided the two-parameter family of rotations used on the
product is replaced by a circle connection of nonzero curvature and the
fibers are made sufficiently small.

\begin{theorem}\label{thm:sphere-bundles}
Let $\Sigma$ be a closed connected oriented surface, and let
$\pi:M\to\Sigma$ be a smooth oriented $S^2$-bundle. Then $M$ admits
a Riemannian metric with positive biorthogonal curvature. The metric
can be chosen invariant under a smooth effective fiberwise $S^1$-action
whose fixed point set consists of two disjoint sections.
\end{theorem}

This includes both oriented bundle types in every genus, and since the
fibers are simply connected, every closed orientable surface group is
realized on each of the two bundle types over the corresponding surface.
The two constructions complement each other. The bundle construction
controls the topology but collapses the fibers as the parameter tends to
zero, whereas the explicit family of Theorem~\ref{thm:main} lives on the
fixed product, retains the full torus symmetry, and converges smoothly to
the prescribed round--flat metric.

The paper is organized as follows. In Section~\ref{sec:preliminaries}, we fix the curvature conventions and
review the three standard lemmas that the constructions rest on: the
Hodge-decomposition formula for $k_g$  and the conformal transformation law. In Section~\ref{sec:connection}, we prove
Theorem~\ref{thm:main}. In Section~\ref{sec:sphere-bundles}, we  prove
Theorem~\ref{thm:sphere-bundles}. Appendix~\ref{app:conformal} contains a
short self-contained proof of the conformal law, and
Appendix~\ref{app:arrow} the eigenvalue bound.

\section{Curvature identities and connection metrics}\label{sec:preliminaries}

Throughout, a closed manifold is compact and has no boundary. Let $(M^4,g)$ be an oriented Riemannian four-manifold, and let $\nabla$
be its Levi--Civita connection. We use the convention
\begin{align*}
 R(X,Y)Z=\nabla_X\nabla_YZ-\nabla_Y\nabla_XZ-\nabla_{[X,Y]}Z.
\end{align*}
The induced inner product on $\Lambda^2T_pM$ is characterized by
\begin{align*}
 \ip{X\wedge Y}{Z\wedge W}
 =g(X,Z)g(Y,W)-g(X,W)g(Y,Z).
\end{align*}
The curvature operator $\Sc:\Lambda^2TM\to\Lambda^2TM$ is the
self-adjoint endomorphism defined by
\begin{align*}
 \ip{\Sc(X\wedge Y)}{Z\wedge W}=g(R(X,Y)W,Z).
\end{align*}
We also write $\Sc(\alpha,\beta)=\ip{\Sc\alpha}{\beta}$ for its
associated symmetric bilinear form. We use $g$ to identify two-vectors
with two-forms, and use the same notation for the induced curvature
operator on the bundle $\Lambda^2T^*M$ of two-forms. If $X,Y$ are orthonormal, then
\begin{align*}
 K_g(\operatorname{span}\{X,Y\})
 =\Sc(X\wedge Y,X\wedge Y).
\end{align*}
In particular, the round unit sphere has sectional curvature one. For an orthonormal frame $e_1,\ldots,e_4$, we abbreviate
\[
 R_{ij,kl}=\Sc_{ij,kl}=\langle R(e_i,e_j)e_l,e_k\rangle,
 \qquad R_{ij,ij}=K_g(e_i,e_j).
\]
Here $K_g(e_i,e_j)$ denotes the sectional curvature of their span.

\subsection{Riemannian submersions}
Let $\pi:(M,g)\to(B,g_B)$ be a Riemannian submersion; see \cite{B1987,ONeill} for more details. Thus
$\mathcal V=\ker d\pi$ and $\mathcal H=\mathcal V^\perp$ are its
vertical and horizontal distributions, and
$d\pi|_{\mathcal H_p}:\mathcal H_p\to T_{\pi(p)}B$ is an isometry
at every point $p$. Denote the corresponding orthogonal projections by
$\mathsf v$ and $\mathsf h$. O'Neill's tensors are
\begin{align*}
 \mathsf A_EF
 &=\mathsf v\nabla_{\mathsf hE}(\mathsf hF)
   +\mathsf h\nabla_{\mathsf hE}(\mathsf vF),\\
 \mathsf T_EF
 &=\mathsf h\nabla_{\mathsf vE}(\mathsf vF)
   +\mathsf v\nabla_{\mathsf vE}(\mathsf hF),
\end{align*}
where $E,F$ are arbitrary vector fields. For horizontal fields $X,Y$
and a vertical field $U$, they satisfy
\begin{align*}
 \mathsf A_XY=\tfrac12\mathsf v[X,Y],\qquad
 \ip{\mathsf A_XU}{Y}=-\ip{U}{\mathsf A_XY}.
\end{align*}
The fibers are totally geodesic if and only if $\mathsf T=0$.

We shall use the following specialization of the submersion formulas
of~\cite{B1987,ONeill}. Suppose that the fibers are totally geodesic. At a
point $p$, let $X,Y$ be orthonormal horizontal vectors and let $U,V$
be orthonormal vertical vectors. Then
\begin{align}\label{eq:oneill}
 K_g(X,Y)&=K_{g_B}(d\pi X,d\pi Y)-3|\mathsf A_XY|^2,\notag\\
 K_g(U,V)&=K_{\pi^{-1}(\pi(p))}(U,V),\\
 K_g(X,U)&=|\mathsf A_XU|^2.\notag
\end{align}
Here $K_g(X,Y)$ denotes the sectional curvature of the plane spanned
by $X,Y$, and each fiber carries its induced metric.

\subsection{The Hodge decomposition and biorthogonal curvature}
Let $\mathbf v_g\in\Lambda^4T_pM$ be the positively oriented unit
four-vector. The Hodge star $*:\Lambda^2T_pM\to\Lambda^2T_pM$ is
defined by
\begin{align*}
 \alpha\wedge *\beta=\ip{\alpha}{\beta}\,\mathbf v_g,
 \qquad \alpha,\beta\in\Lambda^2T_pM.
\end{align*}
It is an isometric involution. Its eigenspaces
$\Lambda_p^\pm=\{\alpha:*\alpha=\pm\alpha\}$ have dimension three
and give the orthogonal decomposition
\begin{align*}
 \Lambda^2T_pM=\Lambda_p^+\oplus\Lambda_p^-.
\end{align*}
If $\{e_i\}_{i=1}^4$ is an oriented orthonormal basis of $T_pM$, then
orthonormal bases of $\Lambda_p^\pm$ are
\begin{align}\label{eq:hodge-basis}
 \omega_1^\pm&=\tfrac1{\sqrt2}(e_1\wedge e_2\pm e_3\wedge e_4),\notag\\
 \omega_2^\pm&=\tfrac1{\sqrt2}(e_1\wedge e_3\mp e_2\wedge e_4),\\
 \omega_3^\pm&=\tfrac1{\sqrt2}(e_1\wedge e_4\pm e_2\wedge e_3).\notag
\end{align}
These pointwise decompositions define the bundles $\Lambda^\pm$.
Write $P_\pm=\tfrac12(I\pm *)$ for the orthogonal projections and set
\begin{align*}
 A_\pm=P_\pm\Sc|_{\Lambda^\pm},\qquad
 B=P_+\Sc|_{\Lambda^-}.
\end{align*}
Thus $A_\pm$ are the diagonal compressions of $\Sc$, and
\begin{align}\label{eq:curvature-blocks}
 \Sc=\begin{pmatrix}A_+&B\\ B^*&A_-\end{pmatrix},\qquad
 A_\pm=W_\pm+\frac{s_g}{12}I_{\Lambda^\pm},
\end{align}
where $s_g$ is scalar curvature, $B^*$ is the adjoint of $B$, and $W_\pm$ are the self-dual and anti-self-dual Weyl curvature operators. This block decomposition goes back to Singer and Thorpe~\cite{ST69}. The off-diagonal block is determined by the traceless Ricci tensor; see~\cite[Section~2]{CR14}.

The following standard formula~\cite[Section~2]{CR14} reduces both constructions to estimates on two symmetric matrices. We record it with our normalization; throughout, $\lambda_{\min}$ denotes the smallest eigenvalue of a self-adjoint operator.

\begin{lemma}[Hodge decomposition]\label{lem:biorthogonal-minimum}\label{lem:hodge-minimum}
For every oriented Riemannian four-manifold $(M,g)$,
\begin{align}\label{eq:min}
 k_g(p)=\frac12\bigl(\lambda_{\min}(A_+(p))
                      +\lambda_{\min}(A_-(p))\bigr),
\end{align}
where  $p\in M$.
\end{lemma}

\subsection{Conformal transformations}
We use $\Delta_g=\operatorname{div}_g\nabla$, so
$\Delta_gf=\operatorname{tr}_g\nabla^2f$. The conformal law below
follows by averaging the complementary-plane sectional curvature formulas in Besse~\cite[Theorem~1.159]{B1987}; the plane-dependent Hessian terms cancel in the average, which is what makes the biorthogonal condition conformally tractable. A short self-contained proof is given in Appendix~\ref{app:conformal}. The law also expresses the role of $12k_g$ in the modified scalar-curvature framework of Costa~\cite{Costa12}.

\begin{lemma}[Conformal change]\label{lem:conformal}
Let $(M^4,g)$ be a Riemannian four-manifold and let $u$ be a positive
smooth function. For every two-plane $\sigma\subset T_pM$,
\begin{align*}
 K_{u^2g}^\perp(\sigma)
 =u^{-2}\left(K_g^\perp(\sigma)-\frac{\Delta_gu}{2u}\right).
\end{align*}
Consequently,
\begin{align}\label{eq:conformal}
 k_{u^2g}=u^{-2}\left(k_g-\frac{\Delta_gu}{2u}\right).
\end{align}
\end{lemma}

\begin{remark}\label{rem:product-conformal}\label{rem:conformal-product}
For the round--flat product $g_0=g_S+g_T$ on $S^2\times T^2$,
all sectional curvatures are nonnegative, and complementary mixed
planes are both flat. Thus $k_{g_0}\equiv0$. If $u>0$ is smooth,
at a minimum point $p$ of $u$ we have $\Delta_{g_0}u(p)\ge0$.
Lemma~\ref{lem:conformal} then gives $k_{u^2g_0}(p)\le0$.
Consequently, no metric conformal to $g_0$ has positive biorthogonal
curvature everywhere.
\end{remark}

\section{An explicit metric on \texorpdfstring{$S^2\times T^2$}{the product of a sphere and a torus}}\label{sec:connection}

The horizontal lifts of the two parallel fields on the torus need not
commute: their vertical bracket contributes a positive term to the
average of the two least Hodge eigenvalues. This term vanishes at the poles
of the sphere, where a conformal correction will supply the missing
curvature. We first establish the estimate that makes this correction
possible. Throughout this section, $g_\eps$ is the metric
in~\eqref{eq:metric}, and $z$ is the height function on the unit sphere.

\begin{proposition}\label{thm:seed}
For every rank-two lattice $\Lambda\subset\R^2$ and
$0<\eps\le1/4$, the metric $g_\eps$ satisfies
\begin{equation}\label{eq:torus-seed}
 k_{g_\eps}(p,t)
 \ge \frac{\eps^4}{8}\bigl(1-z(p)^2\bigr)-\frac{\eps^6}{2},
\end{equation}
where $(p,t)\in S^2\times (\R^2/\Lambda).$
\end{proposition}

\subsection{The connection and its curvature}

Let $\pi:S^2\times T^2\to T^2$ be the projection. Its vertical
space at $(p,t)$ is $\mathcal V_{(p,t)}=T_pS^2\times\{0\}$, and its
horizontal space is the $g_\eps$-orthogonal complement
$\mathcal H_{(p,t)}=\mathcal V_{(p,t)}^{\perp_{g_\eps}}$.
Polarizing~\eqref{eq:metric} gives
\begin{align*}
 g_\eps\bigl((V,a,b),(W,c,d)\bigr)
 &=\ip{V-\eps aX_1-\eps bX_2}{W-\eps cX_1-\eps dX_2}_{g_S}
   +ac+bd.
\end{align*}
Consider the global vector fields
\begin{align*}
 H_1=\partial_{t_1}+\eps X_1,\qquad
 H_2=\partial_{t_2}+\eps X_2.
\end{align*}
In the product splitting they are $(\eps X_1,1,0)$ and
$(\eps X_2,0,1)$. Their sphere components are therefore canceled by the
corresponding terms in the polarized metric. For every vertical vector
$U$ we obtain
\begin{align*}
 g_\eps(H_i,U)=0,\qquad
 g_\eps(H_i,H_j)=\delta_{ij},\qquad i,j\in\{1,2\}.
\end{align*}
Since the horizontal space has dimension two, $H_1,H_2$ form an
orthonormal basis of it. Moreover,
$d\pi(H_i)=\partial_{t_i}$, so $d\pi$ restricts to an isometry from
$\mathcal H_{(p,t)}$ to $(T_tT^2,g_T)$. Thus $\pi$ is a Riemannian
submersion, and $H_i$ is the horizontal lift of $\partial_{t_i}$.

For a vertical field $U$, define its horizontal derivatives and the
curvature field of this splitting by
\begin{equation}\label{eq:torus-horizontal-bracket}
 D_iU=[H_i,U],\qquad F_{ij}=[H_i,H_j],\qquad F=F_{12}.
\end{equation}
These are vertical fields: the lifts $H_i$ project to the commuting
parallel fields $\partial_{t_i}$, while $U$ projects to zero. For
$A\in\R^3$, write $X_A(p)=A\times p$. The identity
$[X_A,X_B]=-X_{A\times B}$ gives
\begin{align*}
 F&=\eps^2[X_1,X_2]=-\eps^2X_3,\\
 D_1F&=-\eps^3[X_1,X_3]=-\eps^3X_2,\\
 D_2F&=-\eps^3[X_2,X_3]=\eps^3X_1.
\end{align*}
In particular, $F_{21}=-F$, whereas $F_{11}=F_{22}=0$.
Writing $p=(x,y,z)\in S^2$ and using $|X_1|^2=1-x^2$, $|X_2|^2=1-y^2$,
and $|X_3|^2=1-z^2$, we have
\begin{align}\label{eq:bounds-F}
 |F|^2=\eps^4(1-z^2),\qquad
 |D_1F|^2+|D_2F|^2=\eps^6(1+z^2)\le2\eps^6.
\end{align}
Here and below, inner products of vertical vectors are taken with the
round metric, which is the restriction of $g_\eps$ to each fiber.

The flow of $H_i$ translates the torus variable and rotates the sphere
about the $E_i$-axis. It therefore maps each round fiber isometrically
onto another fiber. Equivalently, for arbitrary vertical fields $U,V$,
\begin{align*}
 H_i\ip{U}{V}=\ip{D_iU}{V}+\ip{U}{D_iV}.
\end{align*}
This identity will be used both in the connection formula and in the
curvature calculation. Denote by $\nabla^S$ the Levi--Civita connection
along the round fibers and by $\nabla$ the Levi--Civita connection of
$g_\eps$.

The following specialization of the submersion connection formulas
also keeps track of the derivatives of $F$. We record the formulas with
our sign convention; the underlying identities are those of
O'Neill~\cite{ONeill}, as recalled in Section~\ref{sec:preliminaries}.

\begin{lemma}\label{lem:connection}
For vertical vector fields $U,V$ and $i,j\in\{1,2\}$,
\begin{enumerate}
	\item [\normalfont(1)] $\nabla_UV=\nabla^S_UV;$
	\item [\normalfont(2)] $ \nabla_{H_i}H_j=\tfrac12F_{ij};$
	\item [\normalfont(3)] $ \nabla_UH_i=-\tfrac12\sum_{j=1}^2\ip{F_{ij}}{U}H_j;$
	\item [\normalfont(4)] $ \nabla_{H_i}U=D_iU+\nabla_UH_i.$
\end{enumerate}
In particular, the sphere fibers are totally geodesic.
\end{lemma}

\begin{proof}
Horizontal transport preserves the vertical metric, and hence the
fibers are totally geodesic. Indeed, for vertical fields $U,V$, the Koszul
formula together with $[U,H_i]=-D_iU$ and the displayed derivation
identity gives
\[
 2g_\eps(\nabla_UV,H_i)
 =-H_i\ip{U}{V}+\ip{D_iU}{V}+\ip{U}{D_iV}=0.
\]
Since the base fields are parallel,
O'Neill's identities give
$\nabla_{H_i}H_j=\mathsf A_{H_i}H_j=F_{ij}/2$ and
$\ip{\nabla_UH_i}{H_j}=-\ip{U}{F_{ij}}/2$.
Their remaining vertical components vanish because the fibers are
 totally geodesic. The first line (1)
is the induced connection on each fiber, and the last line follows
from the torsion-free identity
$\nabla_{H_i}U-\nabla_UH_i=[H_i,U]$.
\end{proof}

To estimate the least Hodge eigenvalues, we need the mixed curvature
entries as well as the sectional curvatures supplied by
\eqref{eq:oneill}. We now apply the curvature convention of
Section~\ref{sec:preliminaries} to $g_\eps$.
At a fixed point $p$, we choose an oriented orthonormal vertical basis $e_1,e_2$
and put $e_3=H_1$, $e_4=H_2$. The sphere carries its outward orientation,
and the product is oriented by this basis together with $H_1,H_2$.
Define, for $i,a\in\{1,2\}$,
\begin{align*}
 f_a=\ip{F}{e_a},\qquad r^2=f_1^2+f_2^2,\qquad
 d_{ia}=\ip{D_iF}{e_a},\qquad
 q=\ip{\nabla^S_{e_1}F}{e_2}.
\end{align*}
Since $F=-\eps^2X_3$ is Killing on each round fiber, $\nabla^SF$ is
skew-symmetric, so
\begin{align*}
 \nabla^S_{e_1}F=q e_2,\qquad
 \nabla^S_{e_2}F=-q e_1.
\end{align*}
Moreover, for a tangent vector $U$ to the sphere,
$\nabla^S_UF=-\eps^2(E_3\times U)^{\top}$, where $\top$ denotes
orthogonal projection onto $T_pS^2$. It follows that $|q|\le\eps^2$.
In the sequel, we abbreviate
$\ip{\Sc(e_a\wedge e_b)}{e_c\wedge e_d}$ by $\Sc_{ab,cd}$.

\begin{lemma}\label{lem:curvature}
Apart from the symmetry of the curvature operator, its potentially
nonzero entries in the basis
$e_1\wedge e_2,e_1\wedge e_3,e_1\wedge e_4,
 e_2\wedge e_3,e_2\wedge e_4,e_3\wedge e_4$ are
\begin{align}
 \Sc_{12,12}&=1,
 &\Sc_{34,34}&=-\tfrac34r^2,
 &\Sc_{12,34}&=q,\label{eq:R1}\\
 \Sc_{13,13}&=\Sc_{14,14}=\tfrac14f_1^2,
 &\Sc_{23,23}&=\Sc_{24,24}=\tfrac14f_2^2,\label{eq:R2}\\
 \Sc_{13,23}&=\Sc_{14,24}=\tfrac14f_1f_2,
 &\Sc_{13,24}&=\tfrac12q,
 &\Sc_{14,23}&=-\tfrac12q,\label{eq:R3}\\
 \Sc_{34,13}&=-\tfrac12d_{11},
 &\Sc_{34,14}&=-\tfrac12d_{21},\label{eq:R4}\\
 \Sc_{34,23}&=-\tfrac12d_{12},
 &\Sc_{34,24}&=-\tfrac12d_{22}.\label{eq:R5}
\end{align}
All unlisted entries vanish. In particular, every entry with three
vertical arguments and one horizontal argument is zero.
\end{lemma}

\begin{proof}
We calculate the entries by the number and position of their vertical
arguments. All vertical fields used below may depend on both variables
of the product.

\par\medskip\noindent\emph{Case I: Vertical entries.}
For vertical fields $U,V,W$, by Lemma~\ref{lem:connection}, we have
\begin{align*}
 R(U,V)W
 =\nabla^S_U\nabla^S_VW-\nabla^S_V\nabla^S_UW
      -\nabla^S_{[U,V]}W
 =R^S(U,V)W=\ip{V}{W}U-\ip{U}{W}V.
\end{align*}
Thus $\Sc_{12,12}=1$. The vector $R(U,V)W$ is vertical, so the
curvature symmetries also give
\begin{align*}
 \Sc_{12,13}=\Sc_{12,14}=\Sc_{12,23}=\Sc_{12,24}=0.
\end{align*}
\par\medskip\noindent\emph{Case II: Mixed entries.}
Fix $i,j\in\{1,2\}$ and a vertical field $U$. Substituting the connection
formulas into the three terms defining $R(U,H_i)H_j$, we obtain
\begin{align*}
 \nabla_U\nabla_{H_i}H_j
 &=\frac12\nabla^S_UF_{ij},\\
 \nabla_{H_i}\nabla_UH_j
 &=-\frac12\sum_{k=1}^2H_i\ip{F_{jk}}{U}H_k
   -\frac14\sum_{k=1}^2\ip{F_{jk}}{U}F_{ik},\\
 \nabla_{[U,H_i]}H_j
 &=\frac12\sum_{k=1}^2\ip{F_{jk}}{D_iU}H_k.
\end{align*}
Here the last line uses $[U,H_i]=-D_iU$. Since $D_i$ preserves the
vertical inner product,
\begin{align*}
 H_i\ip{F_{jk}}{U}-\ip{F_{jk}}{D_iU}
 =\ip{D_iF_{jk}}{U}.
\end{align*}
It follows that the full mixed curvature vector is
\begin{align*}
 R(U,H_i)H_j
 &=\frac12\nabla^S_UF_{ij}
   +\frac14\sum_{k=1}^2\ip{F_{jk}}{U}F_{ik}
   +\frac12\sum_{k=1}^2\ip{D_iF_{jk}}{U}H_k.
\end{align*}
Taking its inner product with an arbitrary vertical field $V$, we obtain
\begin{align*}
 \ip{R(U,H_1)H_1}{V}
 &=\ip{R(U,H_2)H_2}{V}
   =\frac14\ip{F}{U}\ip{F}{V},\\
 \ip{R(U,H_1)H_2}{V}
 &=\frac12\ip{\nabla^S_UF}{V},\\
 \ip{R(U,H_2)H_1}{V}
 &=-\frac12\ip{\nabla^S_UF}{V}.
\end{align*}
The first line, with $U,V\in\{e_1,e_2\}$, proves~\eqref{eq:R2} and
the first two entries in~\eqref{eq:R3}. The last two lines, with
$U=e_1$ and $V=e_2$, give
$\Sc_{13,24}=q/2$ and $\Sc_{14,23}=-q/2$. With $U=V=e_a$ they give
$\Sc_{13,14}=\Sc_{23,24}=0,$
since $\nabla^SF$ is skew-symmetric. These are the two remaining
unlisted entries between mixed bivectors.

\par\medskip\noindent\emph{Case III: The vertical--horizontal entry.}
For vertical $U,V$, the formulas for $\nabla_UH_i$ yield
\begin{align*}
 \nabla_U\nabla_VH_2
 &=\frac12 U\ip{F}{V}H_1
   -\frac14\ip{F}{V}\ip{F}{U}H_2,\\
 \nabla_V\nabla_UH_2
 &=\frac12 V\ip{F}{U}H_1
   -\frac14\ip{F}{U}\ip{F}{V}H_2,\\
 \nabla_{[U,V]}H_2&=\frac12\ip{F}{[U,V]}H_1.
\end{align*}
The quadratic terms cancel. Metric compatibility and torsion-freeness
of $\nabla^S$, followed by the Killing identity for $F$, give
\begin{align*}
 R(U,V)H_2
 &=\frac12\bigl(U\ip{F}{V}-V\ip{F}{U}-\ip{F}{[U,V]}\bigr)H_1\\
 &=\frac12\bigl(\ip{\nabla^S_UF}{V}
               -\ip{\nabla^S_VF}{U}\bigr)H_1
  =\ip{\nabla^S_UF}{V}H_1.
\end{align*}
Taking $U=e_1$ and $V=e_2$, this  proves
$\Sc_{12,34}=\ip{R(e_1,e_2)H_2}{H_1}=q$.

\par\medskip\noindent\emph{Case IV: Entries involving the horizontal bivector.}
Since $[H_1,H_2]=F$, we have
\begin{align*}
 R(H_1,H_2)H_1
 &=\nabla_{H_1}\!\left(-\tfrac12F\right)
   -\nabla_{H_2}(0)-\nabla_FH_1\\
 &=-\frac12D_1F-\frac32\nabla_FH_1
  =-\frac12D_1F+\frac34r^2H_2,\\
 R(H_1,H_2)H_2
 &=\nabla_{H_1}(0)-\nabla_{H_2}\!\left(\tfrac12F\right)
   -\nabla_FH_2\\
 &=-\frac12D_2F-\frac32\nabla_FH_2
  =-\frac12D_2F-\frac34r^2H_1.
\end{align*}
The horizontal component in the last line gives
$\Sc_{34,34}=-3r^2/4$, and the vertical components give
\begin{align*}
 \Sc_{34,a3}=-\frac12\ip{D_1F}{e_a}=-\frac12d_{1a},\qquad
 \Sc_{34,a4}=-\frac12\ip{D_2F}{e_a}=-\frac12d_{2a},
 \quad a\in\{1,2\}.
\end{align*}
These are~\eqref{eq:R4}--\eqref{eq:R5}. Together with the preceding
calculations and the symmetry of $\Sc$, they account for every entry
of the curvature operator.
\end{proof}

\subsection{The least Hodge eigenvalues}\label{sec:seed}

We apply the Hodge decomposition of Section~\ref{sec:preliminaries} to the
product orientation and the frame $e_1,e_2,H_1,H_2$ of
Section~\ref{sec:connection}. We use the orthonormal bases
$\omega_1^\pm,\omega_2^\pm,\omega_3^\pm$ defined
in~\eqref{eq:hodge-basis}. The notation $A_\pm$ henceforth denotes the
diagonal Hodge blocks of the curvature operator of $g_\eps$.

The terms involving the infinitesimal rotation $\nabla^SF$ have opposite
signs in the two Hodge blocks. Expressing the blocks explicitly shows
that these terms cancel in their averaged least eigenvalue, up to an
error controlled by the horizontal derivatives $D_iF$.

\begin{lemma}\label{lem:hodge}
In the bases $\omega_1^\pm,\omega_2^\pm,\omega_3^\pm$, the diagonal
Hodge blocks of the curvature operator are
\begin{align}\label{eq:hodge}
 A_\pm=\begin{pmatrix}a_\pm&b_\pm^T\\b_\pm&c_\pm I_2\end{pmatrix},
\end{align}
where $a_\pm=\frac12-\frac{3}{8}r^2\pm q,$ $c_\pm=\frac{1}{8}r^2\mp\frac 12 q,$
 and \begin{align*}
 b_+=\frac14\binom{-d_{11}+d_{22}}{-d_{21}-d_{12}},\qquad
 b_-=\frac14\binom{d_{11}+d_{22}}{d_{21}-d_{12}}.
\end{align*}
\end{lemma}

\begin{proof}
Let $s\in\{+1,-1\}$ for the choice of sign, and put
$(A_s)_{ij}=\ip{\Sc\omega_i^s}{\omega_j^s}$.
Expanding the three diagonal entries and using
Lemma~\ref{lem:curvature}, we obtain
\begin{align*}
 (A_s)_{11}
 &=\tfrac12\bigl(\Sc_{12,12}+2s\Sc_{12,34}+\Sc_{34,34}\bigr)
 =\tfrac12-\tfrac38r^2+sq,\\
 (A_s)_{22}
 &=\tfrac12\bigl(\Sc_{13,13}-2s\Sc_{13,24}+\Sc_{24,24}\bigr)
 =\tfrac18(f_1^2+f_2^2)-\tfrac s2q,\\
 (A_s)_{33}
 &=\tfrac12\bigl(\Sc_{14,14}+2s\Sc_{14,23}+\Sc_{23,23}\bigr)
 =\tfrac18(f_1^2+f_2^2)-\tfrac s2q.
\end{align*}
The off-diagonal entry in the lower $2\times2$ block vanishes:
\begin{align*}
 (A_s)_{23}
 &=\tfrac12\bigl(\Sc_{13,14}+s\Sc_{13,23}
       -s\Sc_{24,14}-\Sc_{24,23}\bigr)\\
 &=\tfrac12\bigl(0+\tfrac s4 f_1f_2
       -\tfrac s4 f_1f_2-0\bigr)=0.
\end{align*}
Here $\Sc_{24,14}=\Sc_{14,24}$ by symmetry. Finally, the entries
with three vertical arguments vanish, and
\eqref{eq:R4}--\eqref{eq:R5} yield
\begin{align*}
 (A_s)_{12}
 &=\tfrac12\bigl(\Sc_{12,13}-s\Sc_{12,24}
       +s\Sc_{34,13}-\Sc_{34,24}\bigr)
 =\tfrac14(-sd_{11}+d_{22}),\\
 (A_s)_{13}
 &=\tfrac12\bigl(\Sc_{12,14}+s\Sc_{12,23}
       +s\Sc_{34,14}+\Sc_{34,23}\bigr)
 =\tfrac14(-sd_{21}-d_{12}).
\end{align*}
These entries and symmetry give~\eqref{eq:hodge} and the stated
formulas for $b_+$ and $b_-$.
\end{proof}

\begin{remark}\label{rem:b-bound}
The off-diagonal coefficient vectors satisfy
\begin{align*}
 |b_+|^2+|b_-|^2
 =\frac18\bigl(d_{11}^2+d_{12}^2+d_{21}^2+d_{22}^2\bigr).
\end{align*}
Since $e_1,e_2$ are orthonormal, by \eqref{eq:bounds-F} we have
\begin{align}\label{eq:b-bound}
 |b_+|^2+|b_-|^2
 =\frac{|D_1F|^2+|D_2F|^2}{8}
 =\frac{\eps^6}{8}(1+z^2)\le\frac{\eps^6}{4}.
\end{align}
\end{remark}

\begin{proof}[Proof of Proposition~\ref{thm:seed}]
By Lemma~\ref{lem:biorthogonal-minimum}, it suffices to estimate the least
eigenvalues of the matrices in Lemma~\ref{lem:hodge}.
If $0<\eps\le1/4$, then
\begin{align}\label{eq:gap}
 a_\pm-c_\pm
 =\frac12-\frac{r^2}{2}\pm\frac32q
 \ge\frac12-\frac{\eps^4}{2}-\frac32\eps^2
 \ge\frac12-\frac1{512}-\frac3{32}
   =\frac{207}{512}>\frac14.
\end{align}
The matrices $A_\pm$ of Lemma~\ref{lem:hodge} are arrow matrices in the
sense of Lemma~\ref{lem:arrow}, with $m=2$, head entries $a_\pm$, tail
entries $c_\pm$, and coupling vectors $b_\pm$. By \eqref{eq:gap} their gaps exceed
$1/4$, so Lemma~\ref{lem:arrow} gives
\[
 \lambda_{\min}(A_\pm)\ge c_\pm-\frac{|b_\pm|^2}{a_\pm-c_\pm}.
\]
Since $(c_++c_-)/2=r^2/8$, combining this estimate with
Lemma~\ref{lem:biorthogonal-minimum}, Remark~\ref{rem:b-bound}, and
\eqref{eq:gap} gives
\begin{align*}
 k_{g_\eps}
 &=\frac{\lambda_{\min}(A_+)+\lambda_{\min}(A_-)}2\\
 &\ge\frac{c_++c_-}{2}
       -\frac12\left(\frac{|b_+|^2}{a_+-c_+}
                     +\frac{|b_-|^2}{a_--c_-}\right)\\
 &\ge\frac{r^2}{8}-2(|b_+|^2+|b_-|^2)\\
 &\ge\frac{\eps^4}{8}(1-z^2)-\frac{\eps^6}{2}.
\end{align*}
This completes the proof.
\end{proof}

\subsection{The conformal correction}

The choice of conformal factor is dictated by the leading coefficient in
Proposition~\ref{thm:seed}. Its normalized spherical average and its
mean-zero part satisfy the following identities, where $dA_S$ denotes
the round area form:
\begin{align*}
 \frac1{4\pi}\int_{S^2}\frac{1-z^2}{8}\,dA_S=\frac1{12},\qquad
 \Delta_{S^2}\!\left(\frac{z^2}{24}\right)
 =2\left(\frac{1-z^2}{8}-\frac1{12}\right).
\end{align*}
The conformal law of Lemma~\ref{lem:conformal} therefore suggests a
correction proportional to $z^2$. The factor $1+\eps^2$ in
\eqref{eq:conformal-factor} compensates exactly for the horizontal
contribution to the Laplacian of $g_\eps$, as the following computation
shows.

\begin{proof}[Proof of Theorem~\ref{thm:main}]
For any smooth function $\varphi$ depending only on the sphere variable,
Lemma~\ref{lem:connection} gives
\begin{align*}
 \Delta_{g_\eps}\varphi
 =\Delta_{S^2}\varphi+\eps^2(X_1^2+X_2^2)\varphi.
\end{align*}
Here we have used the frame formula
$\Delta\varphi=\sum_{a=1}^4(e_ae_a\varphi-(\nabla_{e_a}e_a)\varphi)$
and the identities $H_i\varphi=\eps X_i\varphi$.
Writing $p=(x,y,z)$, the rotation fields satisfy
\begin{align*}
 X_1z=y,\qquad X_1y=-z,\qquad
 X_2z=-x,\qquad X_2x=z.
\end{align*}
Consequently,
\begin{align*}
 (X_1^2+X_2^2)z^2
 =2(y^2-z^2)+2(x^2-z^2)=2-6z^2.
\end{align*}
On the unit round sphere, $\Delta_{S^2}z=-2z$ and
$|\nabla^Sz|^2=1-z^2$. The product rule therefore gives
\begin{align*}
 \Delta_{S^2}z^2
 =2z\Delta_{S^2}z+2|\nabla^Sz|^2
 =-4z^2+2(1-z^2)=2-6z^2.
\end{align*}
It follows that
\begin{align}\label{eq:lap}
 \Delta_{g_\eps}z^2=(1+\eps^2)(2-6z^2).
\end{align}
Define
\begin{align*}
 \Phi=\frac{1-z^2}{8},\qquad \overline\Phi=\frac1{12}.
\end{align*}
For $u_\eps$ in~\eqref{eq:conformal-factor}, equation~\eqref{eq:lap} yields
\begin{align*}
 \frac{\Delta_{g_\eps}u_\eps}{2u_\eps}
 =\frac{\eps^4}{u_\eps}(\Phi-\overline\Phi).
\end{align*}
Since $\Phi\ge0$ and $u_\eps\ge1$, Proposition~\ref{thm:seed} gives
\begin{align*}
 k_{g_\eps}-\frac{\Delta_{g_\eps}u_\eps}{2u_\eps}
 &\ge\eps^4\left(\Phi-\frac{\Phi-\overline\Phi}{u_\eps}-\frac{\eps^2}{2}\right)\\
 &=\eps^4\left(\frac{\overline\Phi}{u_\eps}
       +\left(1-\frac1{u_\eps}\right)\Phi-\frac{\eps^2}{2}\right)\\
 &\ge\eps^4\left(\frac{1}{12u_\eps}-\frac{\eps^2}{2}\right).
\end{align*}
Now $1\le u_\eps<2$ and $\eps^2\le1/16$. Thus
\begin{align*}
 \frac{1}{12u_\eps}-\frac{\eps^2}{2}
 \ge\frac1{24}-\frac1{32}=\frac1{96},\qquad
 u_\eps^{-2}\ge\frac14.
\end{align*}
Lemma~\ref{lem:conformal} now proves~\eqref{eq:final-bound}.
All coefficients in~\eqref{eq:metric} and~\eqref{eq:conformal-factor}
are independent of the torus variable, so the metrics are invariant under
torus translations. Their explicit formulas also give convergence to
$g_S+g_T$ in every $C^k$ norm on the compact product as $\eps\to0$.
\end{proof}

\section{Sphere bundles over surfaces}\label{sec:sphere-bundles}

The product construction of the previous section treats one specified
manifold together with its symmetries. In this section we show that the
same conformal method adapts to any prescribed sphere bundle over a
surface, which proves Theorem~\ref{thm:sphere-bundles}.

We first describe the bundle types needed in the construction. For a
complex line bundle $L$, write $L_{\R}$ for its underlying oriented
real plane bundle, and write $S(E)$ for the unit sphere bundle of a
Euclidean vector bundle $E$.

By Smale's theorem~\cite[Theorem~A]{Smale59}, the structure group of an
oriented smooth sphere bundle reduces to $SO(3)$. The classification
of oriented rank-three bundles over a surface by their second
Stiefel--Whitney class therefore gives precisely two oriented
$S^2$-bundles; see~\cite[\S12]{MilnorStasheff}. More explicitly, if $L_e$ is a complex
line bundle of degree $e$, then
\begin{equation}\label{eq:bundle-parity}
 w_2(\R\oplus (L_e)_{\R})=c_1(L_e)\pmod 2.
\end{equation}
Here $c_1$ and $w_2$ denote the first Chern class and the second
Stiefel--Whitney class, respectively. Thus every bundle in the theorem
is isomorphic to $S(\R\oplus(L_e)_{\R})$ with $e=2$ in the trivial
case and $e=1$ in the nontrivial case. Using a nonzero even degree in
the trivial case is useful: it supplies a connection with nonzero
curvature without changing the underlying sphere bundle.

Fix a smooth Riemannian metric $g_\Sigma$ on $\Sigma$, and denote
its area form and Gaussian curvature by $d\mu_\Sigma$ and $K_\Sigma$.
Let $L\to\Sigma$ be a Hermitian line bundle of nonzero degree $e$.
Choose a unitary connection whose real curvature form is
\begin{align}\label{eq:bundle-connection-curvature}
 \Omega=c\,d\mu_\Sigma,\qquad
 |c|=\frac{2\pi|e|}{\operatorname{Area}_{g_\Sigma}(\Sigma)}>0.
\end{align}
To see that such a connection exists, start with a unitary connection
of real curvature $\Omega_0$ and put
$c=\int_\Sigma\Omega_0/\operatorname{Area}_{g_\Sigma}(\Sigma)$.
The form $c\,d\mu_\Sigma-\Omega_0$ has zero integral and hence equals
$d\beta_0$ for a smooth one-form $\beta_0$ on $\Sigma$. Adding
$i\beta_0$ to the connection gives~\eqref{eq:bundle-connection-curvature}.
In particular, $c$ is constant even when $K_\Sigma$ is not.
Its sign depends on the curvature convention and does not affect
the estimates below.

Put $M_L=S(\R\oplus L_{\R})$, and denote its projection by $\pi$.
In a local unitary frame of $L$, write a
point of $M_L$ as $(z,w)\in\R\oplus\mathbb C$, where
$z^2+|w|^2=1$, and write the connection as $d+iA$, so that
$dA=c\,d\mu_\Sigma$. For $\rho>0$,
the associated connection metric is
\begin{align}\label{eq:bundle-metric}
 g_\rho=\pi^*g_\Sigma+\rho^2\bigl(dz^2+|dw+iAw|^2\bigr).
\end{align}
Indeed, under a change of unitary frame,
$w'=e^{-i\chi}w$ and $A'=A+d\chi$, so
\begin{align*}
 dw'+iA'w'=e^{-i\chi}(dw+iAw).
\end{align*}
Thus the local expressions define a global metric. The function $z$
is global because the real line summand is trivial. Writing $w=x+iy$,
the fiber term in~\eqref{eq:bundle-metric} becomes
\begin{align*}
 \rho^2\bigl((dx-yA)^2+(dy+xA)^2+dz^2\bigr),
 \qquad x^2+y^2+z^2=1.
\end{align*}
This expression is smooth at the two pole sections $z=\pm1$.

The projection $\pi:(M_L,g_\rho)\to(\Sigma,g_\Sigma)$ is a Riemannian
submersion with round fibers of radius $\rho$. Horizontal parallel
transport preserves their round metrics. As in Lemma~\ref{lem:connection}, the fibers are therefore
totally geodesic. The horizontal distribution
also preserves $z$.

Shrinking the fibers makes their sectional curvature dominate both
the curvature of the base and the terms linear in the connection
curvature. In the Hodge bases fixed in Section~\ref{sec:preliminaries}, the first diagonal entry is then strictly larger than the two mixed entries, which therefore determine the minimum biorthogonal curvature.

\begin{lemma}\label{lem:bundle-curvature}
If $\rho>0$ satisfies
\begin{align}\label{eq:bundle-gap-condition}
 \rho^{-2}+\min_\Sigma K_\Sigma>\rho^2c^2+3|c|,
\end{align}
then the metric in~\eqref{eq:bundle-metric} satisfies
\begin{align}\label{eq:bundle-seed}
 k_{g_\rho}=\frac{\rho^2c^2}{8}(1-z^2)
\end{align}
on $M_L$.
\end{lemma}

\begin{proof}
\par\medskip\noindent\emph{Step 1: An adapted coframe.}
We compute on the complement of the pole sections. Write
$z=\cos\vartheta$ and $w=\sin\vartheta\,e^{i\phi}$. Then
\begin{align*}
 g_\rho=\rho^2\bigl(d\vartheta^2+
             \sin^2\vartheta(d\phi+A)^2\bigr)+\pi^*g_\Sigma.
\end{align*}
Choose the oriented orthonormal coframe
\begin{align*}
 \theta^1=\rho\,d\vartheta,\qquad
 \theta^2=\rho\sin\vartheta(d\phi+A),\qquad
 \theta^3,\theta^4,
\end{align*}
where $\theta^3,\theta^4$ are pulled back from an oriented orthonormal
coframe on the base. We identify base functions and forms with their
pullbacks in this calculation, and write $K=K_\Sigma\circ\pi$.
Let $\beta$ be the base connection form, so that
\begin{align*}
 d\theta^3=-\beta\wedge\theta^4,\qquad
 d\theta^4=\beta\wedge\theta^3,\qquad
 d\beta=K\,\theta^3\wedge\theta^4.
\end{align*}
Set
\begin{align*}
 \mu=\frac{\cot\vartheta}{\rho},\qquad
 b=\rho c\sin\vartheta,\qquad q=-c\cos\vartheta=-\mu b,
\end{align*}
and abbreviate $\theta^{ij}=\theta^i\wedge\theta^j$. Direct
differentiation gives
\begin{align*}
 d\theta^1=0,\qquad
 d\theta^2=\mu\theta^{12}+b\theta^{34},\qquad
 d\mu=-(\rho^{-2}+\mu^2)\theta^1,\qquad
 db=\mu b\,\theta^1.
\end{align*}
Let $e_1,\ldots,e_4$ be the dual frame, and use connection forms
$\gamma_{ij}$ with
$\nabla e_j=\sum_i\gamma_{ij}e_i$ and
$\gamma_{ij}=-\gamma_{ji}$. The torsion equations
$d\theta^i=-\sum_j\gamma_{ij}\wedge\theta^j$ give
\begin{align*}
 \gamma_{12}&=-\mu\theta^2,
 &\gamma_{13}&=\gamma_{14}=0,\\
 \gamma_{23}&=\tfrac b2\theta^4,
 &\gamma_{24}&=-\tfrac b2\theta^3,\\
 \gamma_{34}&=\beta-\tfrac b2\theta^2.
\end{align*}
\par\medskip\noindent\emph{Step 2: Curvature forms.}
The curvature forms
$\Omega_{ij}=d\gamma_{ij}+\sum_k\gamma_{ik}\wedge\gamma_{kj}$
are therefore
\begin{align}\label{eq:bundle-curvature-forms}
 \Omega_{12}&=\rho^{-2}\theta^{12}+q\theta^{34},
 &\Omega_{13}&=\tfrac q2\theta^{24},\notag\\
 \Omega_{14}&=-\tfrac q2\theta^{23},
 &\Omega_{23}&=-\tfrac q2\theta^{14}+\tfrac{b^2}{4}\theta^{23},\\
 \Omega_{24}&=\tfrac q2\theta^{13}+\tfrac{b^2}{4}\theta^{24},
 &\Omega_{34}&=q\theta^{12}
              +\left(K-\tfrac{3b^2}{4}\right)\theta^{34}.\notag
\end{align}
To make the signs in these formulas explicit, the six differentiations
are as follows. The identities $d\mu=-(\rho^{-2}+\mu^2)\theta^1$,
$db=\mu b\theta^1$, and $q=-\mu b$ give
\begin{align*}
 \Omega_{12}
 &=-d\mu\wedge\theta^2-\mu\,d\theta^2\\
 &=(\rho^{-2}+\mu^2)\theta^{12}
     -\mu^2\theta^{12}-\mu b\theta^{34}
   =\rho^{-2}\theta^{12}+q\theta^{34},\\
 \Omega_{13}
 &=\gamma_{12}\wedge\gamma_{23}
   =-\frac{\mu b}{2}\theta^{24}
   =\frac q2\theta^{24},\\
 \Omega_{14}
 &=\gamma_{12}\wedge\gamma_{24}
   =\frac{\mu b}{2}\theta^{23}
   =-\frac q2\theta^{23}.
\end{align*}
The terms containing the base connection form $\beta$ cancel in the
next two entries:
\begin{align*}
 \Omega_{23}
 &=\frac12db\wedge\theta^4+\frac b2d\theta^4
       +\gamma_{24}\wedge\gamma_{43}\\
 &=\frac{\mu b}{2}\theta^{14}
     +\frac b2\beta\wedge\theta^3
     +\frac b2\theta^3\wedge\beta
     +\frac{b^2}{4}\theta^{23}\\
 &=-\frac q2\theta^{14}+\frac{b^2}{4}\theta^{23},\\
 \Omega_{24}
 &=-\frac12db\wedge\theta^3-\frac b2d\theta^3
       +\gamma_{23}\wedge\gamma_{34}\\
 &=-\frac{\mu b}{2}\theta^{13}
     +\frac b2\beta\wedge\theta^4
     +\frac b2\theta^4\wedge\beta
     +\frac{b^2}{4}\theta^{24}\\
 &=\frac q2\theta^{13}+\frac{b^2}{4}\theta^{24}.
\end{align*}
Finally, the identity $d\beta=K\theta^{34}$ supplies the base-curvature
term in the horizontal entry:
\begin{align*}
 \Omega_{34}
 &=d\beta-\frac12db\wedge\theta^2
          -\frac b2d\theta^2+\gamma_{32}\wedge\gamma_{24}\\
 &=-\frac{\mu b}{2}\theta^{12}-\frac{\mu b}{2}\theta^{12}
     +\left(K-\frac{b^2}{2}-\frac{b^2}{4}\right)\theta^{34}\\
 &=q\theta^{12}+\left(K-\frac{3b^2}{4}\right)\theta^{34}.
\end{align*}
With the
curvature convention of Section~\ref{sec:preliminaries}, they satisfy
$\Omega_{ij}(e_k,e_l)=\Sc_{ij,kl}$.

\par\medskip\noindent\emph{Step 3: The least Hodge eigenvalues.}
In the Hodge bases~\eqref{eq:hodge-basis}, the diagonal blocks of the
curvature operator of $g_\rho$ are consequently
\begin{align*}
 A_\pm=\operatorname{diag}(\alpha_\pm,\gamma_\pm,\gamma_\pm),\qquad
 \alpha_\pm=\frac{\rho^{-2}+K}{2}-\frac{3b^2}{8}\pm q,\qquad
 \gamma_\pm=\frac{b^2}{8}\mp\frac q2.
\end{align*}
For instance, the first two diagonal entries are
\begin{align*}
 \ip{\Sc\omega_1^\pm}{\omega_1^\pm}
 &=\frac12\bigl(\Sc_{12,12}\pm2\Sc_{12,34}+\Sc_{34,34}\bigr)\\
 &=\frac12\left(\frac1{\rho^2}\pm2q+K-\frac{3b^2}{4}\right)
   =\frac{\rho^{-2}+K}{2}-\frac{3b^2}{8}\pm q,\\
 \ip{\Sc\omega_2^\pm}{\omega_2^\pm}
 &=\frac12\bigl(\Sc_{13,13}+\Sc_{24,24}
                    \mp2\Sc_{13,24}\bigr)\\
 &=\frac12\left(0+\frac{b^2}{4}\mp q\right)
   =\frac{b^2}{8}\mp\frac q2.
\end{align*}
The third diagonal entry equals $\gamma_\pm$ by the same expansion
applied to $\omega_3^\pm$ in the basis~\eqref{eq:hodge-basis}, and every
term in an off-diagonal entry vanishes
by~\eqref{eq:bundle-curvature-forms}. Since $b^2\le\rho^2c^2$ and $|q|\le|c|$,
condition~\eqref{eq:bundle-gap-condition} implies
\begin{align*}
 \alpha_\pm-\gamma_\pm
 =\frac{\rho^{-2}+K-b^2\pm3q}{2}
 \ge\frac{\rho^{-2}+\min_\Sigma K_\Sigma-\rho^2c^2-3|c|}{2}>0.
\end{align*}
Thus $\lambda_{\min}(A_\pm)=\gamma_\pm$. By
Lemma~\ref{lem:biorthogonal-minimum},
\begin{align*}
 k_{g_\rho}
 =\frac{\gamma_++\gamma_-}{2}
 =\frac{b^2}{8}
 =\frac{\rho^2c^2}{8}(1-z^2).
\end{align*}
Both sides are continuous on $M_L$, so the equality extends to the
pole sections.
\end{proof}

\begin{proof}[Proof of Theorem~\ref{thm:sphere-bundles}]
\par\medskip\noindent\emph{Step 1: Choice of bundle and fiber size.}
By the classification recalled in~\eqref{eq:bundle-parity}, we may write
$M=S(\R\oplus L_{\R})$ with $L$ of degree $1$ or $2$. Choose a base
metric and a connection as above. Since $\Sigma$ is compact,
condition~\eqref{eq:bundle-gap-condition} holds for all sufficiently
small $\rho>0$.

For such a $\rho$, define
\begin{align}\label{eq:bundle-conformal-factor}
 u_\rho=1+\frac{\rho^4c^2z^2}{24},\qquad
 \widehat g_\rho=u_\rho^2g_\rho.
\end{align}
\par\medskip\noindent\emph{Step 2: The conformal correction.}
We compute the Laplacian of $z^2$. A horizontal lift $H$ satisfies
$Hz=0$, and the vertical component of $\nabla_HH$ is
$\mathsf A_HH=0$. Hence the horizontal terms in the frame formula for
$\Delta_{g_\rho}z^2$ vanish. Since the fibers are totally geodesic, the
remaining terms are their intrinsic Laplacian. The round sphere of
radius $\rho$ therefore gives
\begin{align}\label{eq:bundle-height-laplacian}
 \Delta_{g_\rho}z^2=\rho^{-2}(2-6z^2).
\end{align}
For $u_\rho$ in~\eqref{eq:bundle-conformal-factor}, Lemma~\ref{lem:conformal},
together with~\eqref{eq:bundle-seed} and~\eqref{eq:bundle-height-laplacian},
gives
\begin{align*}
 u_\rho^3k_{\widehat g_\rho}
 &=u_\rho k_{g_\rho}-\tfrac12\Delta_{g_\rho}u_\rho\\
 &=\frac{\rho^2c^2}{8}u_\rho(1-z^2)
       -\frac{\rho^2c^2}{24}(1-3z^2)\\
 &=\rho^2c^2\left(\frac{1-z^2}{8}
                       -\frac{1-3z^2}{24}\right)
       +\frac{\rho^6c^4}{192}z^2(1-z^2)\\
 &=\frac{\rho^2c^2}{12}
       +\frac{\rho^6c^4}{192}z^2(1-z^2).
\end{align*}
Since $1\le u_\rho\le1+\rho^4c^2/24$ and $c\ne0$, this proves
\begin{align*}
 k_{\widehat g_\rho}
 \ge\frac{\rho^2c^2}
          {12(1+\rho^4c^2/24)^3}>0.
\end{align*}
\par\medskip\noindent\emph{Step 3: The circle action.}
Multiplication by $e^{it}$ on $L$ defines a smooth effective
fiberwise circle action. It preserves the connection metric and the
height $z$, hence also $\widehat g_\rho$. Its fixed points are precisely
the two sections $z=1$ and $z=-1$.
\end{proof}

Since the fibers are simply connected, the homotopy exact sequence of
$S^2\longrightarrow M\longrightarrow\Sigma$ gives
$\pi_1(M)\cong\pi_1(\Sigma)$. Thus every closed oriented surface group
is realized on each of the two sphere-bundle types over that surface.

\begin{remark}
For $\Sigma=T^2$ and $\deg L=2$, the sphere bundle is
$S^2\times T^2$, so Theorem~\ref{thm:sphere-bundles} gives another
existence proof on this manifold. The construction in
Theorem~\ref{thm:main} additionally retains the translation action
of $T^2$ and converges smoothly to every prescribed round--flat
product. In the construction of this section, the sphere fibers
collapse as $\rho\to0$.
\end{remark}

Beyond prescribed bundles, it remains an open problem whether every
closed smooth four-manifold admitting positive scalar curvature also
admits positive biorthogonal curvature after at most a change of smooth
structure; compare the simply connected classification
of~\cite{Bettiol17}.

\appendix

\section{The conformal transformation law}\label{app:conformal}

This appendix proves Lemma~\ref{lem:conformal}. The underlying pointwise
formula is the conformal transformation law for sectional curvature; see
Besse~\cite[Theorem~1.159]{B1987}.

\begin{proof}[Proof of Lemma~\ref{lem:conformal}]
Write $u=e^\varphi$ and $\widehat g=u^2g=e^{2\varphi}g$. Fix $p\in M$ and a
two-plane $\sigma\subset T_pM$, and choose a $g$-orthonormal basis
$e_1,\ldots,e_4$ of $T_pM$ with
$\sigma=\operatorname{span}\{e_1,e_2\}$. At the point $p$, the metric
$\widehat g$ is a positive multiple of $g$, so orthogonality is unchanged:
$\sigma^{\perp_{\widehat g}}=\sigma^{\perp_g}
=\operatorname{span}\{e_3,e_4\}$.

For $g$-orthonormal vectors $X,Y$, the conformal transformation of
sectional curvature reads
\begin{align*}
 K_{\widehat g}(X,Y)
 =e^{-2\varphi}\bigl(K_g(X,Y)
 -\nabla^2\varphi(X,X)-\nabla^2\varphi(Y,Y)
 +(X\varphi)^2+(Y\varphi)^2-|\nabla\varphi|_g^2\bigr),
\end{align*}
where $\nabla^2\varphi$ denotes the $g$-Hessian. Apply this formula to the
pairs $(e_1,e_2)$ and $(e_3,e_4)$ and add. The four Hessian terms sum to
$\operatorname{tr}_g\nabla^2\varphi=\Delta_g\varphi$, and the four
first-order terms sum to $|\nabla\varphi|_g^2$, so
\begin{align*}
 K_{\widehat g}(\sigma)+K_{\widehat g}(\sigma^{\perp})
 =e^{-2\varphi}\bigl(K_g(\sigma)+K_g(\sigma^{\perp})
    -\Delta_g\varphi-|\nabla\varphi|_g^2\bigr).
\end{align*}
Since $\Delta_gu=u\,(\Delta_g\varphi+|\nabla\varphi|_g^2)$, dividing by two
gives
\begin{align*}
 K^\perp_{u^2g}(\sigma)
 =u^{-2}\left(K^\perp_g(\sigma)-\frac{\Delta_gu}{2u}\right).
\end{align*}
The minimum over all planes $\sigma\subset T_pM$ is attained for both
metrics on the same set of planes, because the correction term does not
depend on $\sigma$; this proves~\eqref{eq:conformal}.
\end{proof}

\section{The least eigenvalue of an arrow matrix}\label{app:arrow}

Both curvature estimates in this paper reduce, through
Lemma~\ref{lem:biorthogonal-minimum}, to the least eigenvalue of a symmetric
matrix whose off-diagonal part couples one distinguished direction to the
rest. We prove the elementary bound here.

\begin{lemma}[Arrow matrices]\label{lem:arrow}
Let $m\ge1$, let $a,c\in\R$ and $b\in\R^m$, and consider the symmetric
matrix
\[
 A=\begin{pmatrix}a&b^T\\ b&cI_m\end{pmatrix}
\]
of size $m+1$. If $G=a-c>0$, then
\begin{equation}\label{eq:arrow}
 \lambda_{\min}(A)
 =c-\frac{2|b|^2}{G+\sqrt{G^2+4|b|^2}},
 \qquad
 c-\frac{|b|^2}{G}\le\lambda_{\min}(A)\le c.
\end{equation}
\end{lemma}
\begin{proof}
If $b=0$, the matrix is diagonal with entries $a$ and $c$, and
$\lambda_{\min}(A)=c$ because $a>c$; both assertions
of~\eqref{eq:arrow} hold. Suppose $b\ne0$, and split
$\R^m=\R b\oplus b^{\perp}$. The subspace $\{0\}\oplus b^\perp$ consists
of eigenvectors with eigenvalue $c$, and its orthogonal complement
$\R e_0\oplus\R b$, where $e_0$ is the head direction, is invariant, with
matrix
\[
 \begin{pmatrix}a&|b|\\ |b|&c\end{pmatrix}
\]
in the basis $e_0,\,b/|b|$. The eigenvalues of this two-by-two matrix are
\[
 \lambda_\pm=\frac{a+c\pm\sqrt{G^2+4|b|^2}}{2},
\]
and $\lambda_-\le c\le\lambda_+$ because $\sqrt{G^2+4|b|^2}\ge G>0$.
Hence $\lambda_{\min}(A)=\lambda_-$, and rationalizing,
\[
 \lambda_-=c-\frac{\sqrt{G^2+4|b|^2}-G}{2}
 =c-\frac{2|b|^2}{G+\sqrt{G^2+4|b|^2}}.
\]
This is the first assertion of~\eqref{eq:arrow}. The two-sided bound
follows from $G\le G+\sqrt{G^2+4|b|^2}$ and
$\sqrt{G^2+4|b|^2}\ge G$, which give
$0\le c-\lambda_-\le|b|^2/G$.
\end{proof}

\Needspace{9\baselineskip}
\section*{Acknowledgements}
This work was partially supported by NSFC (Grant Nos.\ 12571025, 12131012 and 12301032), NSF of Jiangsu (Grant
No.\ BK20230803), the Fundamental Research Funds for the Central Universities (Grant No.\ 4007012402), and the Zhishan Scholars Program of Southeast University (Grant No.\ 2242024RCB0039).

\medskip
\noindent\textit{Use of generative AI.} ChatGPT was used for exploratory computations and editorial assistance.

\noindent
Zhiqi Chen \\
School of Mathematics and Statistics, Guangdong University of Technology, Guangzhou 510520, P.R. China\\
\textit{E-mail:} \href{mailto:chenzhiqi@nankai.edu.cn}{chenzhiqi@nankai.edu.cn}

\medskip
\noindent
Hui Zhang (corresponding author)\\
School of Mathematics, Southeast University, Nanjing 211189, P. R. China\\
\textit{E-mail:} \href{mailto:huizhang@mail.nankai.edu.cn}{huizhang@mail.nankai.edu.cn}
\end{document}